\documentclass[a4paper,12pt]{article}
\usepackage[english]{babel}
\usepackage{ifthen}
\usepackage{amssymb}
\usepackage[breaklinks]{hyperref}

\makeatletter
\def\@seccntformat#1{\csname the#1\endcsname.\ } 
\makeatother

\makeatletter

\@addtoreset{section}{language}
\makeatother

\date{}

\newenvironment{proof}[1][\hspace{-1.0ex}]%
{\noindent\par\addvspace{1mm}{\sc Proof\hspace{1.0ex}{#1}.} }%
{\quad$\blacktriangle$\par\addvspace{1mm}}

\newif\ifNoRemark
\def\addtheorem#1#2#3#4{
\ifthenelse{\equal{#2}{}}{}%
{\ifthenelse{\expandafter\isundefined\csname the#2\endcsname}{\newcounter{#2}}{}}
\newenvironment{#1}[1][\global\NoRemarktrue]
{\par\addvspace{2mm plus 0.5mm minus 0.2mm}\noindent 
{\bf #3}\ifthenelse{\equal{#2}{}}{}%
{\refstepcounter{#2}{\bf ~\csname the#2\endcsname}}%
{\bf \vphantom{##1}\ifNoRemark.\ \else\ (##1).\fi}\begingroup #4}%
{\endgroup\par\addvspace{1mm plus 0.5mm minus 0.2mm}\global\NoRemarkfalse}
\expandafter\newcommand\csname b#1\endcsname{\begin{#1}}
\expandafter\newcommand\csname e#1\endcsname{\end{#1}}
}

\addtheorem{theorem}{thrm}{Theorem}{\sl}

\makeatletter
\@addtoreset{thrm}{language}
\@addtoreset{lmm}{language}
\@addtoreset{prpstn}{language}
\@addtoreset{problem}{language}
\makeatother

\title{Perfect colorings of the $12$-cube that attain\\ the bound on correlation immunity%
\footnotetext{The work was supported by the RFBR grants 05-01-00816 and 06-01-00694.}%
}
\author{D.~G.~Fon-Der-Flaass}

\begin{document}
\maketitle
\begin{abstract}
 We construct perfect $2$-colorings of the $12$-hypercube that attain our recent bound on the dimension of arbitrary correlation immune functions. We prove that such colorings with parameters $(x,12-x,4+x,8-x$) exist if $x=0$, $2$, $3$ and do not exist if $x=1$.
 
This is a translation into English of the original paper by D.~G.~Fon-Der-Flaass, 
``Perfect colorings of the $12$-cube that attain the bound on correlation immunity'', 
published in Russian in Siberian Electronic Mathematical Reports \cite{0}.
\end{abstract}
Let $H_n$ be the hypercube of dimension $n$. Its vertices are the binary vectors of length $n$ (we will identify such a vector with the set of its nonzero coordinates); two vertices are adjacent if their vectors differ in exactly one coordinate. A coloring of the vertices into black and white colors is called a perfect coloring with parameters $(a,b,c,d)$ if every black vertex has $a$ black and $b$ white neighbors and every white vertex has $c$ black and $d$ white neighbors. (For a general definition of perfect coloring and main properties, see \cite{1}, \cite{3}.)

In \cite{2}, it is proved that for every perfect $2$-coloring of $H_n$ with $b\ne c$, it holds $$c-a\le n/3 $$
(the correlation-immunity bound).
Two series of colorings attaining this bound are known. They are obtained  from a perfect code in the three-dimensional cube and from the Tarannikov coloring with parameters $(1,5,3,3)$ (see \cite{4}) by the product construction \cite[Proposition~1(c)]{1} and have parameters $(0,3k,k,2k)$ and $(k,5k,3k,3k)$, respectively.
Every perfect coloring that attain the correlation-immunity bound has parameters $(i,3x-i,i+x,2x-i)$, the hypercube dimension being $3x$. Without loss of generality we will assume $i<x$. 
In the current work, we determine for which $i$ such colorings of the $12$-dimensional hypercube ($x=4$) exist.

If $i=0$ and $i=2$, then we get parameters that belong to the families mentioned above;  hence colorings exist in these cases.

\begin{theorem} 
There are no perfect colorings of $H=H_{12}$ with parameters $(1,11,5,7)$.
\end{theorem} 
\begin{proof}
Assume, seeking a contradiction, that such a coloring exists.
As in \cite{2}, we define the real-valued function $q$ on $H$ that equals $11$ on the black vertices and $5$ on the white vertices.
It follows from the definition of a perfect coloring that $q$ is an eigenfunction of the adjacency matrix of $H$ with the eigenvalue~$4$.

We will use the approach from \cite{2}. 
Let us remind the notation for the faces of $H$ and the basis $\{f^x\}$ from eigenfunctions.

For $x,y\in H$, $x\cap y = \emptyset$, we define the set $[x]+y=\{z\cup y\,|\,z\subseteq x\}$. This set is called a $k$-face of the hypercube, where $k=|x|$.

For every $x\in H$, the function $f^x$ is defined as 
$$f^x(z)=(-1)^{|z\backslash x|}.$$
The collection $\{f^x\,|\, x\in H\}$ is an orthogonal basis of the space of real-valued functions on $H$. 
Consider the expansion of $q$ in the basis $\{f^x\}$:
$$ q=\sum_{x} w_xf^x $$
where the sum is over all the vectors of weight $4$.
For any vectors $x$, $y$, it is easy to check that 
$\langle \chi^{[x]}, f^y\rangle = 2^{|x|}$ if $x\subseteq y$;
otherwise $\langle \chi^{[x]}, f^y\rangle = 0$.
(Recall that $[x]$ is the smallest face that contains both vertices $x$ and $0$ and $\chi^{[x]}$ is its characteristic function.) From here, we can find the coefficients $w_x$:
$$ \langle \chi^{[x]},q \rangle = 16 w_x = \sum_{v \in [x]} q(v)
= 11m-5(16-m),$$
$$ w_x=-5+m,$$
where $m$ is the number of the black vertices in the face $[x]$. In particular, all the coefficients are integer.

The value $\langle q,q \rangle$ can be calculated in two ways. At first, it equals $2^{12}\sum w_x^2$;
at second, from counting the number of black and white vertices, it equals $5\cdot 2^8\cdot 11^{2}+ 11\cdot 2^8\cdot 5^2 = 55 \cdot 2^{12}$. It follows that $\sum w_x^2 = 55$. 

Hence there are at most $55$ nonzero coefficients. 
Denote $S=\{x \,|\, w_x\ne 0\}$; $|S|\le 55$.

Now consider an arbitrary $3$-face $[y]$. We have:
$$\langle \chi^{[y]},q \rangle = \sum_{[y]\subset[x]}8w_x=11m-5(8-m),$$
$$\sum_{[y]\subset[x]}w_x= 2m-5 \ne 0,$$
where $m$ is the number of the black vertices in the face $[y]$. In particular, this means that each vector $y$ of weight $3$ is contained (as a subset) in at least one vector $x\in S$. Hence $|S| \ge \left({12 \atop 3}\right) /4 = 55$.

So, $|S|=55$, and every weight-$3$ vector is contained in exactly one vector of $S$. But this is impossible, because each of $12$ coordinates must belong to $55\cdot 4/12$ vectors of $S$, which is not an integer. This contradiction proves the theorem.
\end{proof}

\begin{theorem} There exist perfect colorings of $12$-dimensional hypercube with parameters
$(3,9,7,5)$.
\end{theorem} 
\begin{proof} We will give an explicit construction of such colorings, which is, in concept, similar to
the construction from [1].

Let us start the construction from the auxiliary $6$-dimensional cube $X$.
We will denote the coordinates in $X$ by the symbols from the set $\Omega = \{a_1,a_2,a_3,b_1,b_2,b_3\}$.
For the convenience, we will denote an element of $X$ by the list of its nonzero coordinates, omitting
the sign ``$+$'' between them. We will represent the $12$-dimensional hypercube $H$ as the union
of pairwise disjoint layers $L_x$, $x \in X$.
The elements of every layer are also marked by vectors from $X$: $L_x = \{y_x | y\in X\}$
(where $y_x = (y,y+x)$ -- transl. rem.).
The layers $L_x$ are independent sets in $H$.
Two layers $L_x$, $L_{x'}$ with adjacent $x$, $x'$ induce a bipartite graph of order 2;
explicitly, if $x' = x + a$, $a\in \Omega$, then $y_{x'}$ is adjacent with $y_x$ and $(y+a)_x$.

Let us partition all the vertices of $X$ into $4$ black vertices, $12$ white vertices, and
$12$ pairwise disjoint $2$-faces (their vertices will be called grey).
The partition will be invariant with respect to the group $A = \langle \alpha, \beta \rangle$
of automorphisms of $X$ generated by the automorphism of the coordinate permutation
$\alpha = (a_1 a_2 a_3)(b_1 b_2 b_3)$ and the affine automorphism $\beta(x)= a_1 a_2 a_3 + \sigma(x)$
where $\sigma = (a_1 b_1)(a_2 b_2)(a_3 b_3)$.

The black are the four vertices of the orbit $0^A$, namely:
$0$, $a_1 a_2 a_3$, $a_1 a_2 a_3 b_1 b_2 b_3$, $b_1 b_2 b_3$.

The white are the $12$ vertices of the orbit $a_1^A$, namely:
$0+a_i$, $a_1 a_2 a_3 +a_i$, $a_1 a_2 a_3 b_1 b_2 b_3 +a_i$, $b_1 b_2 b_3 +a_i$.

At last, the other vertices of $X$ are partitioned into the $2$-faces that are the images under the action of $A$
of the face $b_1+\langle a_2, a_3\rangle$, namely: \\
\indent  $ b_i+\langle a_j, a_k\rangle$, \\
\indent  $a_j a_k + \langle b_j b_k \rangle$, \\
\indent  $a_1 a_2 a_3 b_j b_k + \langle a_j,a_k \rangle $, \\
\indent  $a_i b_1 b_2 b_3 + \langle b_j,b_k\rangle $, \\
where $i$, $j$, $k$ is an arbitrary permutation of the indices $1$, $2$, $3$.

Now, define a coloring $c:L_x \to \{0,1\}$ for every layer $L_x$.
If $x$ is black, then set $c(L_x)=1$. Similarly,  $c(L_x)=0$ for a white $x$.

For every grey face $G = x + \langle p,q \rangle$ from our partition
(where $x$ is the image of $b_1$ under some automorphism from $A$, and $p$, $q$
are the elements of $\Omega$ that specify the direction of the face),
let $L_G = \{y_z | y\in X, z\in G\}$ be the union of the corresponding layers.
We arbitrarily  choose the value $c(G)=c(0_x)$ for one vertex from $L_x$.
The other values $c(y_z)$ for $y_z \in L_G$ are defined,
starting from this vertex and applying the following rules:

For $r \in \Omega$, set $c((y+r)_z) = c(y_z)$ if $r \in \{ p,q \}$, and $c((y+r)_z)=1-c(y_z)$ otherwise;

set $c(y_{z+r})=1 - c(y_z)$ for any $r\in\{p,q\}$.

It is easy to see that these rules uniquely determine the colors of all the vertices $c(y_z)$ for
$y \in X$, $z \in G$, and that any two adjacent vertices from this set have different colors.

Observe also that

(*) {\em any vertex out of $L_G$ adjacent with $L_G$ has exactly two neighbors in $L_G$, and the colors
of these two neighbors are different.}

Now, it is not difficult to calculate the number of neighbors of colors $0$ and $1$ for each vertex of $X$.

If $z\in X$ is a black vertex, then it has three white and three grey neighbors in $X$.
Every vertex $y_z \in L_z$ has color $1$ and, according to (*), 
has three color-$1$ neighbors and
$3+3 \cdot 2 = 9$ neighbors of color $0$.

If $z\in X$ is a white vertex, then it has one black and five grey neighbors in $X$.
Every vertex $y_z \in L_z$ has color $0$ and, according to (*), 
has five color-$0$ neighbors and
$5+1 \cdot 2 = 7$ neighbors of color $1$.

If $z$ belongs to a grey face $G$, then it has two neighbors inside the face and either
one black, two white and one grey, or one white and three grey neighbors outside the face.
Consequently, each vertex $y_z$ has four neighbors of color different from its color inside $L_G$;
and, calculated using (*) as above, three outside neighbors of color $1$ and five, of color $0$,
independently of the color of $y_z$.

So, each vertex of $H$ has the required number of neighbors of each color, and a coloring
with required parameters is constructed. 
\end{proof}

Finally, we note that the arbitrary choice of the twelve values $c(G)$ in the construction
enables to construct non-isomorphic perfect colorings with parameters $(3,9,7,5)$.
Nevertheless, the number of pairwise non-isomorphic such colorings, as well as the existence of
colorings with parameters $(3,9,7,5)$
that cannot be obtained by our construction, remains unknown.

\end{document}